\documentclass[11pt]{article}
\usepackage[a4paper,margin=2cm]{geometry}
\usepackage[english]{babel}
\usepackage{latexsym,amsmath,enumerate,graphics,enumerate,amsthm,tikz,hyperref,float}
\usepackage[affil-it]{authblk}
\usepackage{enumerate,amsthm,dsfont,pstricks}
\usepackage{latexsym,amsmath,amssymb,amscd,wrapfig,graphicx}
\usepackage{caption}

\newtheorem{theorem}{Theorem}[section]
\newtheorem{lemma}[theorem]{Lemma}
\newtheorem{proposition}[theorem]{Proposition}

\theoremstyle{definition}
\newtheorem{definition}[theorem]{Definition}

\theoremstyle{remark}

\newcommand{\comment}[1]{}

\numberwithin{equation}{section}

\let\epsilon=\varepsilon

\makeatletter
\def\@maketitle{%
  \newpage
  \null
  \vskip 2em%
  \begin{center}%
  \let \footnote \thanks
    {\Large\bfseries \@title \par}%
    \vskip 1.5em%
    {\normalsize
      \lineskip .5em%
      \begin{tabular}[t]{c}%
        \@author
      \end{tabular}\par}%
    \vskip 1em%
    {\normalsize \@date}%
  \end{center}%
  \par
  \vskip 1.5em}
\makeatother

\begin{document}

\title{\sc \huge the horofunction boundary of the infinite dimensional hyperbolic space}

\author{Floris Claassens%
\thanks{Email: \texttt{f.c.claassens@kent.ac.uk; Supported by EPSRC grant EP/M508068/1}}}
\affil{School of Mathematics, Statistics \& Actuarial Science,
University of Kent, Canterbury, CT2 7NX, United Kingdom}

\maketitle
\date{}

\begin{abstract}
In this paper we give a complete description of the horofunction boundary of the infinite dimensional real hyperbolic space, and characterise its Busemann points.
\end{abstract}




\section{Introduction}
In recent years the study of the geometry of the infinite dimensional real hyperbolic space $\mathbb{H}^{\infty}$ has gained significant momentum since it was popularised by Gromov in \cite{GBook}, see \cite{BIM}, \cite{DSUBook}, \cite{DSU}, \cite{LRW1} and \cite{MP}. In this paper we study the horofunction boundary of $\mathbb{H}^\infty$ and its Busemann points. Horofunctions are a fundamental tool in metric geometry and have found applications in numerous fields including, geometric group theory \cite{DKBook}, ergodic theory \cite{GK, KM}, real and complex dynamics \cite{AR, Bea2, CR, Ka, LLNW}, nonlinear operator theory \cite{LNBook, GV1} and metric and non-commutative geometry \cite{Rief}.
The Busemann points were introduced by Rieffel in \cite{Rief}. They are horofunctions that are the limits of so called almost geodesics. These special horofunctions are known to be particularly useful in the studying isometric problems in metric spaces, see for instance \cite{LW, Wal, Wal2}. 
The main goal of this article is to give a complete description of the horofunction boundary and its Busemann points of $\mathbb{H}^{\infty}$. Before we state the result let us briefly recall the hyperboloid model and Klein's model of the infinite dimensional real hyperbolic space.

Let $(H,\langle\cdot,\cdot\rangle)$ be an infinite dimensional Hilbert space and let $V=\mathbb{R}\oplus H$. Let $Q:V\rightarrow\mathbb{R}$ be the quadratic form,
$$Q((\lambda,x))=\lambda^{2}-\langle x,x\rangle$$
for $(\lambda,x)\in V$. The vector space $V$ has a natural cone
$$V_{+}=\{(\lambda,x)\in V:\|x\|<\lambda\}.$$
Let $B:V\times V\rightarrow\mathbb{R}$ be the bilinear form associated with $Q$, 
$$B((\lambda,x),(\mu,y))=\lambda\mu-\langle x,y\rangle$$
for $(\lambda,x),(\mu,y)\in V$. The cone $V_{+}$ can be equipped with a pseudometric $d_{h}:V_{+}\times V_{+}\rightarrow[0,\infty)$ given by
\begin{equation}\label{hyp} 
\cosh(d_{h}(u,v))=\frac{B(u,v)}{\sqrt{Q(u)Q(v)}}
\end{equation}
for $u,v\in V_{+}$, which is a metric between pairs of rays in $V_{+}$. If we restrict $d_{h}$ to a hyperboloid
$$\textbf{H}=\{u\in V_{+}:Q(u)=1\},$$
we obtain $\mathbb{H}^{\infty}=(\textbf{H},d_{h})$ the \textit{hyperboloid model} of the infinite dimensional real hyperbolic space. In this paper we will work with {\em Klein's model }of the infinite dimensional real hyperbolic space $\mathbb{H}^{\infty}=(\textbf{D},d_{h})$, which is defined on the disc
$$\textbf{D}=\{(\lambda,x)\in V:\lambda=1\text{ and }\|x\|<1\}.$$
On $\textbf{D}$ the metric $d_{h}$ coincides with Hilbert's (cross-ratio) metric which is defined as follows. Let $u$ and $v$ be different elements of $\textbf{D}$ and let $l_{u,v}$ be the line through $u$ and $v$. Let $u'$ and $v'$ be the intersection of $l_{x,y}$ and the boundary of $\textbf{D}$ such that $u$ is between $u'$ and $v$ and $v$ is between $u$ and $v'$. Hilbert's metric is given by
$$\delta(u,v)=\frac{1}{2}\log\left(\frac{\|u'-v\|\|v'-u\|}{\|u'-u\|\|v'-v\|}\right)\qquad(u,v\in\textbf{D}).$$
In the study of Hilbert's metric the factor $\frac{1}{2}$ is usually ignored as it plays no role, except for fixing the curvature of the space.

For finite dimensional real hyperbolic spaces $\mathbb{H}^{n}$ it is well known that $\partial\mathbb{H}^{n}$ coincides with the horofunction boundary. In infinite dimensions this is no longer the case. Indeed we will show that in $\mathbb{H}^\infty$ the Busemann points correspond to $\partial \mathbb{H}^\infty$, and that there many horofunctions that are not Busemann points. This phenomenon is caused by the fact that closed balls in $\mathbb{H}^\infty$ are no longer compact. In fact, we will prove the following theorem.

\begin{theorem}\label{Main theorem}
The horofunctions of $\mathbb{H}^{\infty}=(\textbf{D},d_{h})$ are precisely the functions of the form
$$\xi(v)=\log\left(\frac{B(\hat{u},v)+\sqrt{(B(\hat{u},v))^{2}-Q(v)(1-r^{2})}}{(1+r)\sqrt{Q(v)}}\right)\qquad(v\in \textbf{D})$$
where $0<r\leq1$ and $\hat{u}\in\textbf{D}$ such that $0\leq 1-r^{2}<Q(\hat{u})$, or, $r=1$ and $\hat{u}\in\partial\textbf{D}$. Furthermore, $\xi$ is a Busemann point if and only if $r=1$ and $\hat{u}\in\partial\textbf{D}$, in which case
$$\xi(v)=\frac{1}{2}\log\left(\frac{B(\hat{u},v)^{2}}{Q(v)}\right)\qquad(v\in \textbf{D}).$$
\end{theorem}

\section{Preliminaries}
 There is another way to express the hyperbolic distance $d_h$ on $V_+$ by using the so called Birkhoff's version Hilbert's metric, \cite{Birk}, which will be convenient for our purposes. 
This metric is defined in terms of the order structure on $V=\mathbb{R}\oplus H$ induced by the closure of $V_{+}$. Recall that the closed cone $\overline{V}_{+}$ defines a partial order structure on $V$ by $u\leq v$ if $v-u\in\overline{V}_{+}$.

Birkhoff's version of Hilbert's metric $V_+$ uses the \textit{gauge function}
$$M(u/v)=\inf\{\beta>0:u\leq\beta v\}\qquad(u,v\in V_{+})$$
and is defined by
$$\rho(u,v)=\frac{1}{2}\log(M(u/v)M(v/u))\qquad(u,v\in V_{+}).$$
So $d_h=\rho$ on $V_+$, see \cite{LN2} for details. Birkhoff's version of Hilbert's metric was popularised by Bushell \cite{Bush} and can be defined on the interior of any closed cone on a normed space. 

Let us briefly recall the construction of the horofunction boundary. Let $(X,d)$ a metric space and let $b\in X$ be a base point. Consider the natural embedding $i:X\rightarrow C(X)$ given by
$$i(x)(y)=d(x,y)-d(x,b)\qquad(x,y\in X)$$
where $C(X)$ is equipped with the topology of compact convergence; see \cite[\S 46]{MBook}. By the triangle inequality we have
$$|i(x)(y)-i(x)(y')|=|d(y,x)-d(y',x)|\leq d(y,y'),$$
so $i(X)$ is equicontinuous. By the same methods we also find for all $x,y\in X$ that $|i(x)(y)|\leq d(y,b)$, hence $i(X)(y)=\{i(x)(y):x\in X\}$ has compact closure in $\mathbb{R}$. By Ascoli's Theorem; see \cite[Theorem 47.1]{MBook} we find that $i(X)$ has compact closure in $C(X)$. The closure $\overline{i(X)}$ is called the \textit{horofunction compactification} of $(X,d)$. The set $\overline{i(X)}\setminus i(X)$ is called the \textit{horofunction boundary} of $(X,d)$ and its elements are called \textit{horofunctions}. The horofunction compactification is metrizable if $(X,d)$ is proper, in which case each horofunction is the limit of a sequence $(i(x_n))$ in $i(X)$. If, however, $X$ is not proper, horofunctions need not be limits of sequences, but they are limits of nets. 

Inside the horofunction boundary there are special horofunctions called Busemann points, which are limits of so called almost geodesics. 
\begin{definition}
A net $(x_{\alpha})$ in a metric space $(X,d)$ is \textit{almost geodesic} if, for all $\varepsilon>0$ there exists an index $A$ such that for all $\alpha'\geq\alpha\geq A$ we have
$$d(b,x_{\alpha'})\geq d(b,x_{\alpha})+d(x_{\alpha},x_{\alpha'})-\varepsilon.$$
\end{definition}

A horofunction $\xi\in\overline{i(X})\setminus i(X)$ is called a \textit{Busemann point} if there exists an almost geodesic net $(x_{\alpha})$ in $X$ such that $\xi=\lim_{\alpha}i(x_{\alpha})$. 

Note that in finite dimensional hyperbolic space a net $(x_{\alpha})$ in $\textbf{D}$ will only give rise to a horofunction if $(x_{\alpha})$ is unbounded, as $\overline{\textbf{D}}$ is norm compact. For $\mathbb{H}^{\infty}$ there are, however, horofunctions which are limits of bounded nets. Nevertheless the Busemann points are always limits of unbounded nets as observed in \cite{Wal}. For convenience of the reader we have included the proof. 

\begin{proposition}\label{complete Busemann}
Let $(x_{\alpha})$ be a net in a complete metric space $X$. If $(x_{\alpha})$ is almost geodesic and bounded, then $(x_{\alpha})$ converges to some $x\in X$.
\end{proposition}
\begin{proof}
Let $b$ be a base point. The first step is to prove that $d(x_{\alpha},b)$ converges to some $r\in\mathbb{R}$. To see this we define for an index $A$ the supremum $r_{A}=\sup_{\alpha\geq A}d(x_{\alpha},b)$ which exists as the net is bounded. Let $\varepsilon>0$ and let $A$ be an index such that for all $\alpha'\geq\alpha\geq A$ we have
$$d(x_{\alpha'},b)\geq d(x_{\alpha},b)+d(x_{\alpha'},x_{\alpha})-\varepsilon.$$
Let $\alpha_{A}\geq A$ be such that $0\leq r_{A}-d(x_{\alpha_{A}},b)<\varepsilon$. Then for all $\alpha'\geq\alpha_{A}$ we have
$$r_{A}\geq d(x_{\alpha'},b)\geq d(x_{\alpha_{A}},b)+d(x_{\alpha'},x_{\alpha_{A}})-\varepsilon\geq d(x_{\alpha_{A}},b)-\varepsilon\geq r_{A}-2\varepsilon.$$
So for all $\alpha',\alpha\geq\alpha_{A}$ we find that $|d(x_{\alpha'},b)-d(x_{\alpha},b)|\leq 2\varepsilon$.
Hence $(d(x_{\alpha},b))$ is a Cauchy net from which it follows that $\lim_{\alpha}d(x_{\alpha},b)=r$ for some $r\in\mathbb{R}$.\\

Now let $\varepsilon>0$ and let $A$ be an index such that for all $\alpha'\geq\alpha\geq A$ we have $|r-d(x_{\alpha},b)|<\varepsilon$ and
$$d(x_{\alpha'},b)\geq d(x_{\alpha},b)+d(x_{\alpha'},x_{\alpha})-\varepsilon.$$
It follows that
$$d(x_{\alpha'},x_{\alpha})\leq d(x_{\alpha'},b)-d(x_{\alpha},b)+\varepsilon<3\varepsilon$$
Hence $(x_{\alpha})$ is a Cauchy net. The proposition follows by completeness.\\
\end{proof}

\section{Classification of the horofunction boundary of $\mathbb{H}^{\infty}$}
To prove Theorem \ref{Main theorem}, we will first calculate the gauge functions.

\begin{proposition}\label{Gauge spinfactor}
Let $(H,\langle\cdot,\cdot\rangle)$ be a Hilbert space and let $V=\mathbb{R}\oplus H$. For all $(\mu,x),(\gamma,y)\in V_{+}$ we have
$$M((\mu,x)/(\gamma,y))=\frac{\gamma\mu-\langle x,y\rangle+\sqrt{(\gamma\mu-\langle x,y\rangle)^{2}-(\mu^{2}-\|x\|^{2})(\gamma^{2}-\|y\|^{2})}}{\gamma^{2}-\|y\|^{2}}.$$
\end{proposition}
\begin{proof}
We know that
\begin{align*}
M((\mu,x)/(\gamma,y))&=\inf\{\beta>0:(\mu,x)\leq\beta(\gamma,y)\}\\
&=\inf\{\beta>0:(\gamma\beta-\mu)^{2}\geq\|\beta y-x\|^{2}\text{ and }\gamma\beta-\mu\geq0\}.
\end{align*}
So we have to solve
$$(\gamma\beta-\mu)^{2}-\|\beta y-x\|^{2}=(\gamma^{2}-\|y\|^{2})\beta^{2}-2(\gamma\mu-\langle x,y\rangle)\beta+(\mu^{2}-\|x\|^{2})=0,$$
which has solutions
$$\beta_{\pm}=\frac{\gamma\mu-\langle x,y\rangle\pm\sqrt{(\gamma\mu-\langle x,y\rangle)^{2}-(\mu^{2}-x^{2})(\gamma^{2}-\|y\|^{2})}}{\gamma^{2}-\|y\|^{2}}.$$
Note though, that
\begin{align*}
\gamma\beta_{-}-\mu&=\gamma\frac{\gamma\mu-\langle x,y\rangle-\sqrt{(\gamma\mu-\langle x,y\rangle)^{2}-(\mu^{2}-\|x\|^{2})(\gamma^{2}-\|y\|^{2})}}{\gamma^{2}-\|y\|^{2}}-\mu\\
&=\gamma\frac{(\gamma\mu-\langle x,y\rangle)^{2}-(\gamma\mu-\langle x,y\rangle)^{2}+(\mu^{2}-\|x\|^{2})(\gamma^{2}-\|y\|^{2})}{(\gamma\mu-\langle x,y\rangle+\sqrt{(\gamma\mu-\langle x,y\rangle)^{2}-(\mu^{2}-\|x\|^{2})(\gamma^{2}-\|y\|^{2})})(\gamma^{2}-\|y\|^{2})}-\mu\\
&=\frac{\gamma(\mu^{2}-\|x\|^{2})}{\gamma\mu-\langle x,y\rangle+\sqrt{(\gamma\mu-\langle x,y\rangle)^{2}-(\mu^{2}-\|x\|^{2})(\gamma^{2}-\|y\|^{2})}}-\mu\\
&\leq\frac{\gamma(\mu^{2}-\|x\|^{2})}{\gamma\mu-\|x\|\|y\|+\sqrt{(\gamma\mu-\|x\|\|y\|)^{2}-(\mu^{2}-\|x\|^{2})(\gamma^{2}-\|y\|^{2})}}-\mu\\
&=\frac{\gamma(\mu^{2}-\|x\|^{2})}{\gamma\mu-\|x\|\|y\|+|\mu\|y\|-\gamma\|x\||}-\mu\\
&=\frac{\mu\|x\|\|y\|-\gamma\|x\|^{2}-|\mu^{2}\|y\|-\mu\gamma\|x\||}{\gamma\mu-\|x\|\|y\|+|\mu\|y\|-\gamma\|x\||}.
\end{align*}
We find that if $\mu\|y\|<\gamma\|x\|$, then clearly $\gamma\beta_{-}-\mu<0$. If $\mu\|y\|\geq\gamma\|x\|$, then consider
\begin{align*}
\gamma\beta_{-}-\mu&\leq\frac{\mu\|x\|\|y\|-\gamma\|x\|^{2}-\mu^{2}\|y\|+\mu\gamma\|x\|}{\gamma\mu-\|x\|\|y\|+\mu\|y\|-\gamma\|x\|}\\
&=\frac{(\mu\|y\|-\gamma\|x\|)(\|x\|-\mu)}{(\gamma+\|y\|)(\mu-\|x\|)}=-\frac{\mu\|y\|-\gamma\|x\|}{\gamma+\|y\|}\leq 0.
\end{align*}
Hence we find that $M((\mu,x)/(\gamma,y))=\beta_{+}$.\\
\end{proof}

For all $u,v\in V_{+}$ we can rewrite this result using the quadratic and bilinear forms as
$$M(u/v)=\frac{B(u,v)+\sqrt{B(u,v)^{2}-Q(u)Q(v)}}{Q(v)}.$$
Note that if $Q(u)=Q(v)=1$, then using Proposition \ref{Gauge spinfactor} we find
\begin{align*}
d_{h}(u,v)&=\log(B(u,v)+\sqrt{B(u,v)^{2}-1})\\
&=\cosh^{-1}(B(u,v)),
\end{align*}
which shows that indeed on $V_{+}$ the hyperbolic metric $d_{h}$ coincides with Birkhoff's version of the Hilbert metric. We also need the following basic result from functional analysis.

\begin{lemma}\label{convergence}
Let $(x_{\alpha})$ be a net in a Hilbert space $H$ such that $x_{\alpha}$ converges in the weak topology to some $x\in H$ and $(\|x_{\alpha}\|)$ converges to some $r\geq0$. Then $r\geq\|x\|$. Moreover, if $r=\|x\|$, then $(x_{\alpha})$ converges to $x$ in the norm topology.
\end{lemma}
\begin{proof}
Note that 
$$\|x\|^{2}=\lim_{\alpha}|\langle x,x_{\alpha}\rangle|\leq\lim_{\alpha}\|x\|\|x_{\alpha}\|=r\|x\|.$$
Now suppose that $r=\|x\|$. Then
$$\lim_{\alpha}\|x-x_{\alpha}\|^{2}=\lim_{\alpha}\|x\|^{2}+\|x_{\alpha}\|^{2}-2\langle x,x_{\alpha}\rangle=0.$$
\end{proof}

Using this we can now characterise the horofunctions of $\mathbb{H}^{\infty}$ as follows.

\begin{theorem}\label{horofunctions spinfactor}
Let $(H,\langle\cdot,\cdot\rangle)$ be an infinite dimensional Hilbert space and let $V=\mathbb{R}\oplus H$. The horofunctions of $\mathbb{H}^{\infty}=(\textbf{D},d_{h})$ are the functions of the following form: 
$$\xi((\gamma,y))=\log\left(\frac{\gamma-\langle \hat{x},y\rangle+\sqrt{(\gamma-\langle \hat{x},y\rangle)^{2}-(\gamma^{2}-\|y\|^{2})(1-r^{2})}}{(1+r)\sqrt{\gamma^{2}-\|y\|^{2}}}\right)\qquad((\gamma,y)\in V)$$ 
where either $\|\hat{x}\|<1$ and $\|\hat{x}\|<r\leq1$ or $\|\hat{x}\|=r=1$. 
\end{theorem}
\begin{proof}
Let $((1,x_{\alpha}))$ be a net in $V_{+}$ such that $d_{h}(\cdot,(1,x_{\alpha}))-d_{h}((1,0),(1,x_{\alpha}))$ converges to a horofunction. By taking a subnet we may assume that $(x_{\alpha})$ weakly converges to some $\hat{x}\in H$ as the unit ball is weakly compact and $(\|x_{\alpha}\|)$ converges to some $r\leq1$. Note that by Lemma \ref{convergence}, $r\geq\|\hat{x}\|$. Let $(\gamma,y)\in V_{+}$. Using Proposition \ref{Gauge spinfactor} we find
\begin{align*}
M((1,x_{\alpha})/(1,0))&=1+\|x_{\alpha}\|\\
M((1,0))/(1,x_{\alpha}))&=\frac{1+\|x_{\alpha}\|}{1-\|x_{\alpha}\|^{2}}\\
M((1,x_{\alpha})/(\gamma,y))&=\frac{\gamma-\langle x,y\rangle+\sqrt{(\gamma-\langle x,y\rangle)^{2}-(1-\|x_{\alpha}\|^{2})(\gamma^{2}-\|y\|^{2})}}{\gamma^{2}-\|y\|^{2}}\\
M((\gamma,y))/(1,x_{\alpha}))&=\frac{\gamma-\langle x,y\rangle+\sqrt{(\gamma-\langle x,y\rangle)^{2}-(1-\|x_{\alpha}\|^{2})(\gamma^{2}-\|y\|^{2})}}{1-\|x_{\alpha}\|^{2}}.
\end{align*}
Hence
\begin{align*}
i((1,x_{\alpha}))((\gamma,y))=&\frac{1}{2}\log(M((\gamma,y)/(1,x_{\alpha}))M((1,x_{\alpha})/(\gamma,y)))\\
&\qquad-\frac{1}{2}\log(M((1,0)/(1,x_{\alpha}))M((1,x_{\alpha})/(1,0)))\\
=&\log\left(\frac{\gamma-\langle x_{\alpha},y\rangle+\sqrt{(\gamma-\langle x_{\alpha},y\rangle)^{2}-(1-\|x_{\alpha}\|^{2})(\gamma^{2}-\|y\|^{2})}}{\sqrt{\gamma^{2}-\|y\|^{2}}\sqrt{1-\|x_{\alpha}\|^{2}}}\right)\\
&\qquad-\log\left(\frac{1+\|x_{\alpha}\|}{\sqrt{1-\|x_{\alpha}\|^{2}}}\right)\\
=&\log\left(\frac{\gamma-\langle x_{\alpha},y\rangle+\sqrt{(\gamma-\langle x_{\alpha},y\rangle)^{2}-(1-\|x_{\alpha}\|^{2})(\gamma^{2}-\|y\|^{2})}}{(1+\|x_{\alpha}\|)\sqrt{\gamma^{2}-\|y\|^{2}}}\right).
\end{align*}
Taking the limit gives us 
$$\xi((\gamma,y))=\log\left(\frac{\gamma-\langle \hat{x},y\rangle+\sqrt{(\gamma-\langle \hat{x},y\rangle)^{2}-(\gamma^{2}-\|y\|^{2})(1-r^{2})}}{(1+r)\sqrt{\gamma^{2}-\|y\|^{2}}}\right).$$
Note that if $r=\|\hat{x}\|<1$, then $\xi=i(1,\hat{x})$. So $r>\|\hat{x}\|$, if $\|\hat{x}\|<1$.

Now suppose that a function is of the form as described above. Note that all we need to do is find a net $((1,x_{\alpha}))$ in $V_{+}$ such that $(x_{\alpha})$ converges weakly to $\hat{x}$ and $(\|x_{\alpha}\|)$ converges to $r$. Then it will give rise to the desired horofunction by the above. If $\|\hat{x}\|=1$, consider the sequence $((1,(1-\frac{1}{n})\hat{x}))$, clearly this sequence converges strongly to $(1,\hat{x})$ and gives rise to a horofunction by the above. If $\|\hat{x}\|<1$, then let $(e_{n})$ be an orthonormal sequence in $H$, which exists as $\dim(H)=\infty$, and consider the sequence $((1,\hat{x}+\sqrt{r^{2}-\|\hat{x}\|^{2}}e_{n}))$. Note that $(\hat{x}+\sqrt{r^{2}-\|\hat{x}\|^{2}}e_{n})$ converges weakly to $\hat{x}$, since $(e_{n})$ converges weakly to 0. Also note that 
$$\lim_{n\rightarrow\infty}\|\hat{x}+\sqrt{r^{2}-\|\hat{x}\|^{2}}e_{n}\|^{2}=\lim_{n\rightarrow\infty}r^{2}+2\sqrt{r^{2}-\|\hat{x}\|^{2}}\langle\hat{x},e_{n}\rangle=r^{2}.$$
\end{proof}

Note that the proof of Theorem \ref{horofunctions spinfactor} also shows that $\xi$ is a Busemann point if $\|\hat{x}\|=1$. We can show that these are the only horofunctions that are Busemann points.

\begin{theorem}\label{Busemann points spinfactor}
Let $(H,\langle\cdot,\cdot\rangle)$ be an infinite dimensional Hilbert space and let $V=\mathbb{R}\oplus H$, let $\hat{x}\in H$ and $\|\hat{x}\|\leq r\leq1$ and let 
$$\xi((\gamma,y))=\log\left(\frac{\gamma-\langle \hat{x},y\rangle+\sqrt{(\gamma-\langle \hat{x},y\rangle)^{2}-(\gamma^{2}-\|y\|^{2})(1-r^{2})}}{(1+r)\sqrt{\gamma^{2}-\|y\|^{2}}}\right)$$
be a horofunction. Then $\xi$ is a Busemann point if and only if $\|\hat{x}\|=r=1$.
\end{theorem}
\begin{proof}
In Theorem \ref{horofunctions spinfactor} we already proved that if $\|\hat{x}\|=r=1$, then $\xi$ is a Busemann point. Now suppose that $\xi$ is a Busemann point and let $((1,x_{\alpha}))$ be an almost geodesic net such that $i((1,x_{\alpha}))$ converges to $\xi$. Combining Proposition \ref{complete Busemann} and Theorem \ref{horofunctions spinfactor} gives us that $d_{h}((1,0),(1,x_{\alpha}))$ is not bounded, so $\lim_{\alpha}\|x_{\alpha}\|=r=1$. Note that we can rewrite the horofunction as 
$$\xi((1,y))=\log\left(\frac{1-\langle\hat{x},y\rangle}{\sqrt{1-\|y\|^{2}}}\right).$$
Now suppose $\|\hat{x}\|<1$. Let $\epsilon>0$ and let $A$ be such that for all $\alpha'\geq\alpha\geq A$ we have
$$\varepsilon+d_{h}((1,0),(1,x_{\alpha'}))\geq d_{h}((1,0),(1,x_{\alpha}))+d_{h}((1,x_{\alpha}),(1,x_{\alpha'}))$$
As in the proof of Theorem \ref{horofunctions spinfactor}, using Proposition \ref{Gauge spinfactor} we find
\begin{align*}
\varepsilon\geq&\log\left(\frac{1+\|x_{\alpha}\|}{\sqrt{1-\|x_{\alpha}\|^{2}}}\right)-\log\left(\frac{1+\|x_{\alpha'}\|}{\sqrt{1-\|x_{\alpha'}\|^{2}}}\right)\\
&\qquad+\log\left(\frac{1-\langle x_{\alpha},x_{\alpha'}\rangle+\sqrt{(1-\langle x_{\alpha},x_{\alpha'}\rangle)^{2}-(1-\|x_{\alpha}\|^{2})(1-\|x_{\alpha'}\|^{2})}}{\sqrt{1-\|x_{\alpha'}\|^{2}}\sqrt{1-\|x_{\alpha}\|^{2}}}\right).
\end{align*}
Taking the exponential we find
$$e^{\varepsilon}\geq\frac{1-\langle x_{\alpha},x_{\alpha'}\rangle+\sqrt{(1-\langle x_{\alpha},x_{\alpha'}\rangle)^{2}-(1-\|x_{\alpha}\|^{2})(1-\|x_{\alpha'}\|^{2})}}{(1-\|x_{\alpha}\|)(1+\|x_{\alpha'}\|)}.$$
As this holds for all $\alpha'\geq\alpha$, we can take the limit with respect to $\alpha'$ to get
\begin{align*}
e^{\varepsilon}&\geq\lim_{\alpha'}\frac{1-\langle x_{\alpha},x_{\alpha'}\rangle+\sqrt{(1-\langle x_{\alpha},x_{\alpha'}\rangle)^{2}-(1-\|x_{\alpha}\|^{2})(1-\|x_{\alpha'}\|^{2})}}{(1-\|x_{\alpha}\|)(1+\|x_{\alpha'}\|)}\\
&=\frac{1-\langle x_{\alpha},\hat{x}\rangle}{1-\|x_{\alpha}\|}.
\end{align*}
Finally, as this holds for all $\alpha\geq A$, we can take the limit with respect to $\alpha$ to find
$$e^{\varepsilon}\geq\lim_{\alpha}\frac{1-\langle x_{\alpha},\hat{x}\rangle}{1-\|x_{\alpha}\|}=\infty,$$
which is a contradiction.\\
\end{proof}

Theorem \ref{Main theorem} follows from Theorem \ref{horofunctions spinfactor} and Theorem \ref{Busemann points spinfactor}.

\textit{Acknowledgements.} I would like to thank Bas Lemmens for the advice and help with writing this article.

\end{document}